\documentclass[a4paper]{article}

\usepackage[utf8]{inputenc}
\usepackage{lmodern}
\usepackage{amssymb,amsfonts,amsthm,amsmath,bbm,mathrsfs,aliascnt,mathtools}
\usepackage[english]{babel}
\usepackage[colorlinks]{hyperref}
\usepackage{tikz}
\usetikzlibrary{decorations}
\usetikzlibrary{decorations.pathreplacing}
\tikzset{individu/.style={draw,thick}}

\theoremstyle{plain}
\newtheorem{theorem}{Theorem}[section]
\newtheorem{corollary}[theorem]{Corollary}
\newtheorem{lemma}[theorem]{Lemma}

\newtheorem{remark}[theorem]{Remark}

\numberwithin{equation}{section}

\newcommand{\N}{\mathbb{N}}
\newcommand{\Z}{\mathbb{Z}}
\newcommand{\R}{\mathbb{R}}

\newcommand{\ind}[1]{\mathbf{1}_{\left\{#1\right\}}}

\newcommand{\e}{\mathrm{e}}
\newcommand{\dd}{\mathrm{d}}

\DeclareMathOperator{\E}{\mathbb{E}}

\renewcommand{\P}{\mathbb{P}}

\renewcommand{\epsilon}{\varepsilon}

\title{A model for an  epidemic with contact tracing and cluster isolation, and a detection paradox}
\author{Jean Bertoin\thanks{Institute of Mathematics, University of Zurich, Switzerland.} }
\date{ }

\begin{document}

\maketitle

\begin{abstract} We determine the distributions of some random variables related to a simple model of an epidemic with contact tracing and cluster isolation. This enables us to apply general limit theorems for super-critical Crump-Mode-Jagers branching processes. Notably,   we compute explicitly  the asymptotic proportion of isolated clusters with a given size amongst all isolated clusters, conditionally on survival of the epidemic. 
Somewhat surprisingly, the latter differs from the distribution of the size of a typical cluster at the time of its detection; and we explain the reasons behind this seeming paradox. 
\end{abstract}


\medskip


\section{Introduction}

Predicting and controlling the evolution of epidemics has motivated mathematical contributions for a long time and generated a huge literature; let us merely point at the lecture notes \cite{BritPar} and references therein. Models involving contact tracing and isolation, which aims at reducing the transmissibility of infections, have raised a significant interest; see notably among others \cite{Balletal, Bansaye, Barlow, Huo, Lambert, MK, Metal, OkMu}. 
We present below a toy model in this framework, which is clearly oversimplified\footnote{Many important aspects such as  the possibility of recovery, the age-dependency of the contamination rate, the spacial locations and displacements of infected individuals,... are not taken into account.} and likely unrealistic for practical applications, but which is solvable in the sense that many quantities of interest can be computed explicitly. This model is close to the one introduced  recently by Bansaye, Gu, and Yuan \cite{Bansaye} as it will be discussed in the final section of this text.

We take only into account infected individuals,  implicitly assuming that there  is an infinite reservoir of healthy individuals susceptible to become infected at some point.
There is no death nor recovery, but we distinguish between the contagious individuals and those who have been isolated and hence have ceased to spread the epidemic.
The infected population grows with time as new individuals are contaminated; we suppose that a newly infected individual is always contaminated by a single contagious individual. Imagine further that when a contamination occurs, it can be either \textit{traceable}, for instance in the case of 
a contamination between two relatives, or \textit{untraceable}, for instance in the case when it occurs during a public event between two unrelated individuals. 
At any time, there is thus a natural partition of the infected population into \textit{clusters}, where two individuals are parts of the same cluster if and only if  the contamination path between those individuals can be fully traced.
Last, we suppose that individuals are randomly tested, and when a contagious individual is detected, then one isolates instantaneously its entire cluster. A newly infected individual is always contagious until it has been isolated and then ceases to contaminate further individuals forever.  We stress the distinction between detection, which acts on individuals, and isolation, which follows from detection of a contagious individual and applies to a whole cluster. Clusters consisting of contagious individuals are called active, and then isolated after detection of an infected individual.

We now turn this model into a simple stochastic evolution depending on three parameters, namely: 
\begin{itemize}
\item $\gamma >0$,  the contamination rate of a contagious individual,
\item $p\in(0,1)$, the probability   of traceability for a contamination event,
\item $\delta>0$,  the rate of detection for a contagious individual.
\end{itemize}
 In words, the probability that a contagious  individual at time $t$ contaminates some healthy individual
during the time interval $[t, t+ \dd t]$ is $\gamma \dd t$, and when this occurs, the probability that the contamination is traceable is $p$. Simultaneously, the probability that a contagious individual is detected during the time interval $[t, t+ \dd t]$ is $\delta \dd t$. We suppose that these events are mutually independent, simultaneously for all times and all contagious individuals. In particular, the probability that an active cluster of size $s$ at time $t$ is put into isolation during the time interval $[t, t+ \dd t]$  is $s\delta \dd t$. Last, we suppose for simplicity that at the initial time $t=0$, there is a single infected individual in the population, which we call the ancestor. 

The epidemic eventually  stops once all contagious individuals have been isolated, and we shall see that this occurs almost surely if and only if
the rate of detection is greater than or equal to the rate of untraceable contaminations, i.e. $\delta\geq (1-p)\gamma$. Note that this is independent of the rate $p\gamma$ of traceable contaminations. 
We are mostly interested in the super-critical case $\delta<(1-p)\gamma$ when the epidemic survives forever with strictly positive probability. 

Our main results in the super-critical regime specify  to our setting general limit theorems for Crump-Mode-Jagers branching processes. They show that the number of active, respectively, isolated, clusters counted with some characteristic grows exponentially fast in time with exponent $\alpha=\alpha(\gamma,p,\delta)$ given by the Malthusian parameter. The limits after rescaling involve  as universal factor (that is independent of the chosen characteristics) the terminal value of the so-called intrinsic martingale. As a consequence, conditionally on survival of the epidemic,  the empirical distribution of the sizes of active  clusters (respectively of isolated clusters) converges as time goes to infinity. More precisely, we will show that  the proportion of clusters of given size $k\geq 1$ amongst all active clusters at time $t$ tends to
\begin{equation} \label{E:prea}  c_a(1-\delta/\rho)^{k-1} {\mathrm B} (1+\alpha/\rho, k)
\end{equation}
 as $t\to \infty$,  whereas this proportion  amongst isolated clusters at time $t$ tends to
\begin{equation} \label{E:empisol}
 c_i(1-\delta/\rho)^{k-1} {\mathrm B} (\alpha/\rho, k+1),
 \end{equation}
 where $\rho=\delta + p \gamma$, $\mathrm B$ denotes the beta function, and $c_a$ and $c_i$ are the normalization factors.

Concretely,   the only observable variables at a given time in this model are  the isolated clusters, since, by definition, the active ones have not yet been detected.
Our results point at the following
rather surprising feature (at least for non-specialists of general branching processes or of structured population models). As, loosely speaking, isolated clusters are independent with the same distribution, one might expect that when the epidemic has spread for a long time, the empirical distribution of the isolated clusters should be close
to the law of a typical isolated cluster, that is the cluster generated by a typical contagious individual at the time when it is isolated.  However, it is easy to see that the size of a typical isolated cluster has the geometric distribution with success parameter $\delta/\rho$, so the probability that  a typical isolated cluster has size $k$
equals 
$$ \frac{\delta}{\rho}(1-\delta/\rho)^{k-1},$$
which differs from \eqref{E:empisol}. This is the detection paradox alluded to in the title of this work and  which will be explained in the last section. 
Note that the bias factor ${\mathrm B} (\alpha/\rho, k+1)$ in \eqref{E:empisol} decays as $k$ increases, 
which entails that in the large time limit,  the empirical isolated cluster is in fact stochastically smaller than the typical cluster. 

The plan of this text is as follows. In Section 2, we explain how the model can be recasted in terms of a Crump-Mode-Jagers branching process by focussing on the clusters. In Section 3, we describe the evolution of a typical cluster as a Yule process stopped at the time when it exceeds an independent geometric variable. This enables us to  derive a number of related statistics explicitly, in particular regarding the point process of untraceable contaminations which are induced. Our main results on the large time behavior of the epidemic in the super-critical regime are presented in Section 4; they are merely deduced from well-known general limit theorem for Crump-Mode-Jagers branching processes using the explicit formulas of Section 3. In Section 5, we first compare the model by Bansaye, Gu and Yuan and the approach there with ours. We notably observe that the Malthusian parameter $\alpha$ and the limiting distribution \eqref{E:prea} solve a natural eigenproblem when the evolution of active cluster sizes is viewed as an age-structured population model.
We then explain the detection paradox, and finally, we briefly  discuss the relation between \eqref{E:prea} and the classical Yule-Simon distribution.

\section{The Crump-Mode-Jagers branching process of clusters} 
Although we introduced the epidemic model from the perspective of individuals, it will be convenient for its analysis to rather look at clusters and their evolutions as time passes.
Specifically, imagine that we  create an (unoriented) edge between the infector and the infected at the time when 
a contamination occurs; each edge is further labelled traceable or untraceable, depending on the type of the contamination. 
If we ignore the labels of edges, this endows the  infected population at any given time with a genealogical  tree structure which is rooted at ancestor. 
Plainly this tree structure grows as new individuals are contaminated and new traceable or untraceable edges are added. The reader may find Figure \ref{fig:cluster} above useful to visualize the notions that will be introduced. 

Any pair of infected individuals is connected by a unique segment in the tree, which we call contamination path. Two individuals 
belong to the same cluster if and only if their contamination path contains only traceable edges, and more generally, the number of untraceable edges along the contamination path 
between two individuals only depends on the two clusters to which these individuals  belong. We then define the generation of a given cluster as the number of untraceable edges along the contamination path between  any individual in that cluster and the ancestor.

\begin{figure}
\begin{center}
\includegraphics[height=6cm]{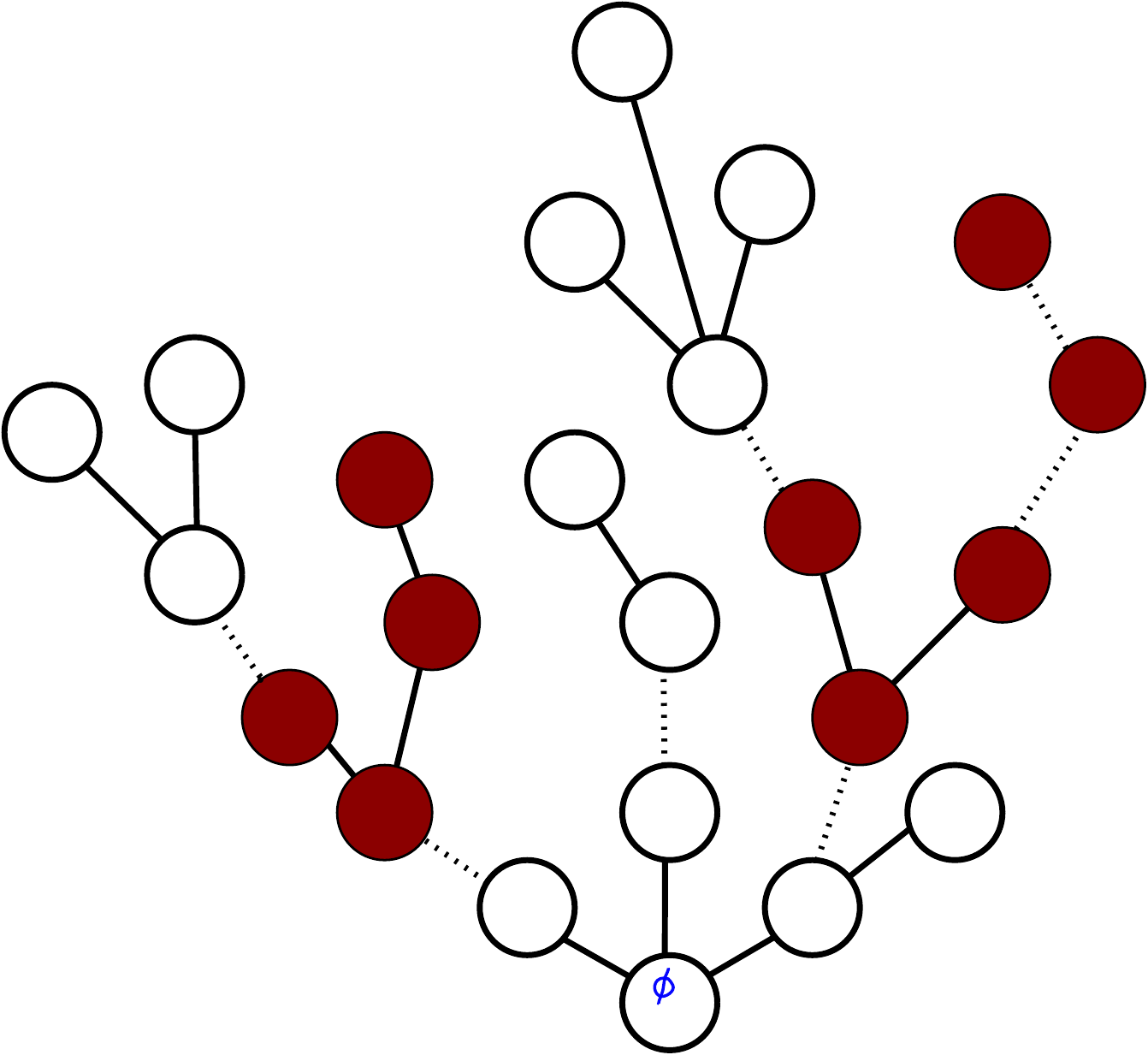}
\end{center}
\caption{ Graphical representation of the epidemic at a given time. The ancestor is the vertex at the bottom of the figure.  Vertices in red represent  contagious individuals, vertices in white  individuals who have been isolated. Full edges indicate traceable contaminations, and dotted ones untraceable contaminations. 
Clusters consist of subsets of vertices connected by full edges. In turn clusters are connected by dotted edges. There are four active clusters: two at the first generation  with sizes $4$ and $3$, one at the third generation with size $1$, and one at the fourth generation with size $1$. There are four isolated clusters: the ancestor cluster with sizes $5$, one cluster with size $2$ at the first generation, and two clusters with sizes $4$ and $3 $ at the second generation. }
\label{fig:cluster}
\end{figure}

We next observe that backtracking contaminations endows  clusters  with a natural genealogy, which in turn  enables us to view the epidemic  model as a so-called Crump-Mode-Jagers branching process\footnote{The arguments  in this section  are rather robust, in the sense that they remain valid for more sophisticated  versions of the model. For instance, one could incorporate recovery, let the contamination rates depend on the age of the infection, etc. However, the quantitative results in the next section are  much more fragile; notably the calculations for the key Lemma \ref{L1} there can not be adapted even to deal with recovery or death. 
}; see  \cite[Chapter 6]{JagersB} as well as \cite{Jagers, Nerman, JagersNerman} for classical background,  and also \cite[Section 5]{HolJan} for a more recent survey with further references. For this purpose, we shall index each cluster  by a finite sequence of positive integers, that is by a vertex $u$ of the Ulam-Harris tree ${\mathcal U}=\bigcup_{n=0}^{\infty} \N^n$, such that the length $|u|$ of $u$ corresponds to the generation of the cluster. By convention, the empty sequence $\varnothing$ with length $0$ is used to label the cluster containing the ancestor. Clusters at the first generation are those such that there is a single untraceable edge along the contamination path from an infected individual in this cluster to the ancestor. They are indexed by $\N^1= \N=\{1,2, \ldots\}$ according to the increasing order of their birth times, that is times at which an individual in the ancestral cluster causes an untraceable contamination and generates a new cluster. For the sake of definitiveness,
  we agree that when the ancestral cluster generates only $k$ untraceable contaminations until it is isolated, then the clusters indexed by $k+1, k+2, \ldots$ are fictitious clusters born at time $\infty$.  This is only a formality, and of course we shall only be concerned with non-fictitious clusters.  By an obvious iteration, we label clusters at the $n$-th generation by $u=(u_1, \ldots, u_n)\in \N^n$ for any $n\geq 0$. It should be plain that the genealogy of clusters does not change with time, in the sense that once a cluster is born, its label will remain the same in the future.

For any $u\in {\mathcal U}$, if the cluster  labelled by $u$ is not fictitious, then we write $\zeta_u$ for the age (that is the time elapsed from the birth-time) at which this cluster is isolated. We also write $\xi_u$ for the simple point process on $[0,\infty)$ of the ages  at which this cluster is involved into untraceable contaminations, that is, generates new clusters. So $\xi_u([0,t])$ is the number of children clusters generated when the cluster reaches age $t$, and in particular $\xi_u([\zeta_u,\infty))=0$. 
Finally, we write $C_u=(C_u(t), t\geq 0)$, where $C_u(t)$ is the size  of this cluster at time $t<\zeta_u$ and we agree that $C_u(t)=0$ whenever $t\geq \zeta_u$. In words,  $C_u$ is the process of 
 the number of contagious individuals in that cluster as a function of its age (recall that infected individuals are no longer contagious once they have been isolated); in particular $C(\zeta-)$ is the size reached by the cluster at the time when it is detected. 
 The most relevant information about the evolution of clusters and hence about the epidemic is encoded by the family of pairs
 $${\mathbf C}_u=(C_u, \xi_u), \qquad u\in \mathcal U,$$
 where we agree for definitiveness that $C_u\equiv 0$ and $\xi_u\equiv 0$ when the cluster indexed by $u$ is fictitious. 
Of course, ${\mathbf C}_u$ does not enable us to recover the subtree structure of the cluster indexed by $u$, but this is irrelevant for the questions we are interested in. 

 It should be intuitively clear that the distribution of the ancestral cluster ${\mathbf C}_{\varnothing}$ determines that of the whole process $\left({\mathbf C}_u\right)_{u\in \mathcal U}$. 
 More precisely, let us write ${\mathbf C}=(C, \xi)$ for a pair distributed as ${\mathbf C}_{\varnothing}$, which  we think of as describing the evolution of a \textit{typical cluster}.
Then it is readily checked that conditionally on $\xi_{\varnothing}(\R_+)=k$, ${\mathbf C}_1, \ldots, {\mathbf C}_{k}$ are $k$ i.i.d. copies of ${\mathbf C}$ which are further independent of ${\mathbf C}_{\varnothing}$. More generally, it follows by iteration that for every $n\geq 1$ and any  $u_1, \ldots, u_k\in \N^{n}$, conditionally on  the event that none of the clusters  ${\mathbf C}_{u_1}, \ldots, {\mathbf C}_{u_k}$ are fictitious (which is measurable with respect to the family
$\left({\mathbf C}_v: |v|<n\right)$ ), ${\mathbf C}_{u_1}, \ldots, {\mathbf C}_{u_k}$ are $k$ i.i.d. copies of ${\mathbf C}$ which are further independent of $\left({\mathbf C}_v: |v|<n\right)$. In other words, $\left({\mathbf C}_u\right)_{u\in \mathcal U}$ generates a Crump-Mode-Jagers branching process where the evolution of typical  elements is distributed as ${\mathbf C}$.

\begin{remark} \label{R:1}
 If we interpret the isolation time $\zeta$ of an active cluster  as the death-time, and if we further view the size $C(t)$ at time $t$ as measuring some ``age'' of the cluster, in the loose sense that this quantity grows with time until death occurs, then we are essentially in the framework of  age-structured population models; see for instance \cite[Section II.E]{Kot}. This aspect will be useful in the forthcoming Section 5.1. In this area, we further refer to \cite{Huo} for a different model for contact tracing in an epidemic in terms of a disease age structured population. 
\end{remark}

\section{Statistics of a typical cluster}
We discuss here some basic  statistics of the typical cluster ${\mathbf C}=(C, \xi)$ in terms of the parameters $(\gamma,p,\delta)$ of the model. Recall that 
the integer valued process $C$ is absorbed at $0$ at the time 
 $$\zeta=\inf\{t\geq 0: C(t)=0\}$$ 
 when this cluster is detected and isolated, and that $\xi$ is the point process of times at which  non-traceable contaminations occur.

It is convenient to set now
\begin{equation} \label{E:rho}
\rho \coloneqq \delta +p\gamma,
\end{equation}
and recall that a Yule process with rate $\rho>0$ refers to a pure birth process with birth rate $\rho k$ from any state $k\geq 1$ and  started from $1$. 

\begin{lemma}\label{L1} The process $(C(t), t\geq 0)$ has the same law as
$$(\ind{Y(t)\leq G} Y(t), t\geq 0),$$
where
$Y=(Y(t), t\geq 0)$ is a Yule process with rate $\rho$, and 
$G$ an independent geometric variable with success probability $\delta/\rho$, viz. with tail distribution function
$$
\P(G > k)=  (1-\delta/\rho)^{k}  , \qquad k\geq 0.$$
\end{lemma}
\begin{proof}
The process $C$ is a continuous time Markov chain on $\Z_+=\{0,1, \ldots\}$, which starts from $1$ at time $0$ and is absorbed at the cemetery state $0$. 
Recall that only traceable contaminations contribute to the growth of the cluster, that they occur at rate $p\gamma$ per contagious individual, 
and that each individual in the cluster is detected at rate $\delta$.

We see that when the chain is at some state $k\geq 1$, its next jumps occurs
after a waiting time with the exponential distribution with parameter $k(p\gamma + \delta)= k\rho$, and independently of this waiting time,  the state after that jump is $k+1$ with probability $p\gamma/\rho$ and is $0$ with complementary probability $\delta/\rho$.
In particular, the size reached by the cluster when it is isolated is a
geometric variable with success probability $\delta/\rho$. Our claim follows from the classical properties of independent exponential variables. 
\end{proof}

Lemma \ref{L1} shows in particular that the size $C(\zeta-)$ of the typical isolated cluster has the geometric distribution with success probability $\delta/\rho$.
The one-dimensional marginal laws of the typical cluster size  process as well as  the joint distribution 
 of the time of isolation $\zeta$ and  $C(\zeta-)$  follow readily.
\begin{corollary} \label{C1} For every $t\geq 0$, we have
$$\P(C(t)=k)=  (1-\delta/\rho)^{k-1} (1-\e^{-\rho t})^{k-1} \e^{-\rho t}\qquad \text{for } k\geq 1,$$
and
$$ \P(C(t)=0)= \P(\zeta \leq t )=1- \frac{\rho}{\rho +\delta(\e^{\rho t}-1)}.$$
Furthermore, we have also
$$\P(C(\zeta-)=k, t\geq \zeta)= \frac{ \delta}{\rho}(1-\delta/\rho)^{k-1} (1-\e^{-\rho t})^k   \qquad \text{for } k\geq 1.$$
\end{corollary}
\begin{proof}
It suffices to write for $k\geq 1$ that
\begin{align*}
\P(C(t)=k)
&= \P(Y(t)=k , G\geq k) \\
&= \P(Y(t)=k) \P(G \geq k),
\end{align*}
and recall that $Y(t)$ has the geometric distribution  with success probability $\e^{-\rho t}$. 
Then summation for $k\geq 1$ yields the second formula of the statement.
We get the third formula similarly, writing  for $k\geq 1$ that
$$
\P(C(\zeta-)=k, t\geq \zeta)= \P(G=k,  Y(t)>k)
= \P(G = k) \P(Y(t)> k).
$$
\end{proof}

We next turn our attention  to the point process $\xi$ at which new clusters are generated, and write
$$Z_1\coloneqq \xi(\R_+)$$
for the total number of (non-fictitious) clusters that the typical cluster begets. 
Its distribution is obtained by a slight  variation of the argument for Lemma \ref{L1}
and this entails the criterion for extinction of the epidemic that has been announced in the introduction. 

\begin{lemma}\label{L:7}
The variable $1+Z_1$ follows the geometric distribution with success probability $\frac{\delta}{(1-p) \gamma + \delta}$. In particular, $Z_1\in L^r(\P)$ for all $r\geq 1$, 
$$ \E(Z_1)= (1-p) \gamma / \delta ,$$
and as a consequence, the total number of infected individuals is finite (in other words, the epidemics eventually ceases) almost surely if and only if
$$(1-p)\gamma \leq \delta.$$
\end{lemma}

\begin{proof}
Fix some arbitrary time $t\geq 0$, and work conditionally on the event that at time  $t$, the typical cluster has size $k\geq 1$ and is still active. Consider the first event after time $t$ at which either
there is a new traceable or untraceable contamination, or the cluster is detected. 
The probability that this event is due to an untraceable contamination is $(1-p)\gamma/(\gamma+ \delta)$, whereas probability that this event is due to detection is $\delta/(\gamma+ \delta)$. 
In the remaining case, the size of the cluster increases by one unit.

The  probabilities above depend neither on $t$ nor $k$, and it follows by iteration that if we now introduce  the first instant  $\tau$ after $t$ at which either 
an untraceable contamination occurs or the cluster is detected, then independently of $C(\tau)$,
the probability that $\tau$ is the time of an untraceable contamination equals 
$\frac{(1-p)\gamma}{(1-p) \gamma + \delta}$ (this is the failure probability). 
Another iteration yields our first claim, and the formula for the first moment of $Z_1$ follows.

Finally, 
if we write $Z_n$ for the number of (non-fictitious) clusters at the $n$-th generation, then 
$(Z_n, n\geq 0)$ is a Galton-Watson process with reproduction law distributed according to $Z_1$. 
So if $\delta < (1-p)\gamma$, there is a strictly positive probability that this Galton-Watson process survives for ever, in which case the 
total number of infected individuals is obviously infinite. Otherwise, the Galton-Watson process becomes eventually extinct almost-surely, there are only finitely many 
(non-fictitious) clusters, each of which consisting of finitely many infected individuals. 
\end{proof}

Last, we introduce the intensity measure $\mu$ of the point process $\xi$,
$$\mu(t) \coloneqq \E(\xi([0,t])), \qquad t\geq 0.$$

\begin{corollary}\label{C2} For every $t\geq 0$, there is the identity
$$  \mu(t) =   (1-p)\frac{\gamma}{\delta} \left(1- \frac{ 1}{1+\delta(\e^{\rho t}-1)/\rho}\right).$$
\end{corollary}

\begin{proof} Indeed, the conditional probability given the process $C$ that a non-traceable contamination event occurs during the time-interval $[t, t+\dd t]$ equals
$(1-p)\gamma C(t)\dd t$, and as a consequence,
$$\dd  \mu(t) =  (1-p)\gamma\E(C(t))\dd t. $$
We deduce from Corollary \ref{C1} that 
$$\E(C(t))=  \frac{ \e^{\rho t}}{(1+\delta(\e^{\rho t}-1)/\rho)^2}.$$
The formula in the statement follows. 
\end{proof}
We note that letting $t\to \infty$ in Corollary \ref{C2} yields
$$\E(\xi(\R_+))=(1-p)\gamma/\delta,$$
in agreement with Lemma \ref{L:7}.

\section{The Malthusian behavior}
We shall assume throughout this section that 
\begin{equation} \label{E:surcrit} \delta < (1-p)\gamma,
\end{equation}
so that the epidemic survives with strictly positive probability.
More precisely, one immediately deduces from Lemma \ref{L:7} that the probability of extinction equals $\delta/((1-p)\gamma)$, which is the smallest solution to the equation $\E(x^{Z_1})=x$.
We shall derive here the main results of this work, simply by specifying in our setting some fundamental results of  Nerman \cite{Nerman} on the asymptotic behavior of Crump-Mode-Jagers branching processes with random characteristics.
 We   start by introducing some of the key actors in this framework. 

Consider the Laplace transform of the intensity measure of untraceable contaminations for a typical cluster,
$$
{\mathcal L}(x)= \int_0^{\infty} \e^{-xt} \dd \mu(t), \qquad x\geq 0.
$$
Since ${\mathcal L}(0)=(1-p)\gamma/\delta >1$, the equation $ {\mathcal L}(x)=1$
possesses a unique solution $\alpha=\alpha(\gamma,p,\delta)\in(0,\infty)$, called the \textit{Malthusian parameter}. That is, thanks to Corollary \ref{C2},
$$(1-p)\gamma \int_0^{\infty}   \frac{ \e^{(\rho-\alpha) t}}{(1+\delta(\e^{\rho t}-1)/\rho)^2}\dd t = 1,$$
or equivalently, in a slightly simpler form, using the change of variables $x=\e^{-\rho t}$,
\begin{equation} \label{E:alpha}
(1-p)\gamma \rho \int_0^1   \frac{ x^{\alpha/\rho}}{((\rho-\delta) x + \delta)^2}\dd x = 1.
 \end{equation}
We  further set
\begin{equation} \label{E:beta}
\beta = - {\mathcal L}'(\alpha)= (1-p)\gamma \int_0^{\infty}   t \frac{ \e^{(\rho-\alpha) t}}{(1+\delta(\e^{\rho t}-1)/\rho)^2}\dd t;
\end{equation}
plainly $\beta\in (0,\infty)$. 

Next, it is convenient to use the notation
$$\langle m,f\rangle \coloneqq \sum_{n=1}^{\infty} f(n) m(n),$$
where $m=(m(n), n\in \N)$ is a finite measure  on $\N$ 
and  $f: \N\to \R_+$ a generic nonnegative function.
We introduce two important measures $m^a$ and $m^i$ related to typical active and isolated clusters  respectively, by
$$\langle m^a,f\rangle =\int_0^{\infty} \e^{-\alpha t} \E(f(C(t)), t<\zeta) \dd t ,$$
and
$$ \langle m^i,f\rangle = \int_0^{\infty} \e^{-\alpha t} \E(f(C(\zeta-)), \zeta\leq t ) \dd t .$$
These two measures can be determined explicitly from Corollary \ref{C1}, using the notation
$$\mathrm B(x,y)= \frac{\Gamma(x) \Gamma(y)}{\Gamma(x+y)}=\int_0^1 s^{x-1}(1-s)^{y-1}\dd s , \qquad x,y>0,$$
for the beta function. Indeed, we then obtain from the change of variables $\e^{-\rho t} = s$, that for every $k\geq 1$:
\begin{align} \label{E:ma}
m^a(k) &= (1-\delta/\rho)^{k-1} \int_0^{\infty} \e^{-\alpha t} (1-\e^{-\rho t})^{k-1} \e^{-\rho t} \dd t 
\nonumber \\
&= \frac{1}{\rho} (1-\delta/\rho)^{k-1} {\mathrm B} (1+\alpha/\rho, k),
\end{align}
and 
\begin{align} \label{E:mi}
m^i(k) &=
\frac{\delta}{\rho} (1-\delta/\rho)^{k-1} \int_0^{\infty} \e^{-\alpha t} (1-\e^{-\rho t})^{k} \dd t
\nonumber \\
&= \frac{\delta}{\rho^2} (1-\delta/\rho)^{k-1} {\mathrm B} (\alpha/\rho, k+1). 
\end{align}

Finally, we introduce  
$$W_n=\sum_{u\in \N^n} \e^{-\alpha \sigma_u}, \qquad n\geq 0,$$
where  $\sigma_u$ stands for the birth-time of the cluster labelled by $u$ (so that $\sigma_u=\infty$ and $\e^{-\alpha \sigma_u}=0$ if this cluster is fictitious). The process $(W_n, n\geq 0)$ is a  martingale, often referred to as the intrinsic martingale; see Jagers \cite[Chapter 6]{JagersB}). Using the inequality $W_1\leq \xi(\R_+)$ and Lemma \ref{L:7}, we see that
 $$\E(W_1^2)<\infty$$ and the
uniform integrability of the intrinsic martingale follows, see e.g. \cite[Theorem 6.1]{Jagers}. We furthermore recall that its 
 terminal value $W_{\infty}$ is strictly positive on the event that the epidemic survives, and of course $W_{\infty}=0$ on the event that the epidemic eventually ceases. 

For the sake of simplicity, we focus on few natural  statistics of the epidemic at time $t\geq 0$. Given a generic nonnegative function $f: \N\to \R_+$, we agree implicitly  that $f(0)=0$ and write
$$ A^f(t) = \sum f(C_u(t-\sigma_u)),$$
where the sum is taken over all vertices $u$ in the Ulam-Harris tree $ \mathcal U$ such that the cluster labelled by $u$ is born at time $\sigma_u\leq t$ (note that only active clusters at time $t$ contribute to the sum). 
Turning our attention to isolated clusters, we write similarly 
$$ I^f(t) =\sum f(C_u(\zeta_u-))\ind{\sigma_u+\zeta_u\leq t},$$
where the sum is taken over all clusters  which are isolated at time $t$.
In the Crump-Mode-Jagers terminology, $A^f$ and $I^f$ are known as 
the processes counted with the random characteristics 
\begin{equation}\label{E:randchar}
\phi^a: t\mapsto f(C(t)) \ind{t<\zeta} \quad\text{and} \quad \phi^i: t\mapsto f(C(\zeta-)) \ind{t\geq\zeta},
\end{equation}
respectively.

Recall the notation above, and notably \eqref{E:alpha},
 \eqref{E:beta}, \eqref{E:ma} and  \eqref{E:mi}. We can now state the main result of this work. 

 \begin{theorem}\label{T1} Assume \eqref{E:surcrit} and let $f: \N\to \R_+$ with $f(n)=O(\e^{b n})$ for some $b< -\log(1-\delta/\rho)$. The following limits then hold almost surely and in $L^1(\P)$: 
$$\lim_{t\to \infty} \e^{-\alpha t} A^f(t) =  \beta^{-1} \langle m^a, f \rangle W_{\infty},$$
and
$$ \lim_{t\to \infty} \e^{-\alpha t} I^f(t)= \beta^{-1} \langle m^i, f \rangle W_{\infty}.$$
\end{theorem}
In particular, taking $f(n)=n$ yields the first order asymptotic behavior of the total number of contagious (respectively, isolated) individuals as time goes to infinity.
\begin{proof} The claim of almost sure convergence is seen from Theorem 5.4  in Nerman \cite{Nerman}; we just need to verify Conditions 5.1 and 5.2 there.
 For the first, we simply write
$$\int_t^{\infty} \e^{-\alpha s} \xi(\dd s) \leq \e^{-\alpha t} \xi(\R_+),$$
and recall from Lemma \ref{L:7} that $Z_1=\xi(\R_+)$ is integrable. This ensures that
$$\E\left( \sup_{t\geq 0} \e^{\alpha t} \int_t^{\infty} \e^{-\alpha s} \xi(\dd s) 
\right) < \infty.$$

For the second, we assume for simplicity that $|f(n) |\leq \e^{b n}$ for all $n\geq 1$ without loss of generality. The
 random characteristics  in \eqref{E:randchar} can be bounded for all $t\geq 0$ by
 $$|\phi^a(t)| \leq \exp(b C(\zeta-)) \quad \text{and} \quad |\phi^i(t)| \leq \exp(b C(\zeta-)).$$
 Observe that 
 $$\E\left( \exp(b C(\zeta-))\right) < \infty,$$
 since we know from Lemma \ref{L1} that $C(\zeta-)$ has the geometric distribution with success probability $\delta/\rho$,
 and $(1-\delta/\rho) \e^b <1$.
It follows immediately that Condition 5.2 of \cite{Nerman} holds for both $\phi^a$ and $\phi^i$. 

We next turn our attention to convergence in $L^1(\P)$. This relies in turn on
 \cite[Corollary 3.3 ]{Nerman} and we have to check Equations 3.1 and 3.2 there.
The latter are both straightforward from the bounds established in the first part of this proof.
\end{proof}
\begin{remark} Iksanov, Kolesko and Meiners \cite{IKM} have obtained recently a remarkable central limit theorem for  general Crump-Mode-Jagers branching processes counted with random characteristics, which specifies the fluctuations of Nerman's law of large numbers. Of course, it would be interesting to apply their results to our setting; however in order to do so, one needs information about the possible roots to the equation $\mathcal L(z)=1$ in the complex strip $\alpha/2\leq z\leq \alpha$, which does not seem easy to obtain even though the intensity $\mu$ is explicitly known. \end{remark}

We next turn our attention to the empirical distribution of the sizes of active, respectively, isolated, clusters. We denote the function identical to $1$ on $\N$ by $\mathbf 1$, so that 
in the preceding notation, 
$A^{\mathbf 1}(t)$
and $I^{\mathbf 1}(t)$ are respectively the number of active and of isolated clusters at time $t$. 
We then define the empirical distributions $\Pi^a(t)$ and $\Pi^i(t)$  for a generic function $f: \N\to \R_+$ by 
$$\langle \Pi^a(t), f\rangle = A^f(t)/A^{\mathbf 1}(t)$$
and
$$\langle \Pi^i(t), f\rangle = I^f(t)/I^{\mathbf 1}(t).$$
We also 
introduce the normalized probability measures  on $\N$
$$
\pi^a = m^a/ \langle m^a, \mathbf 1\rangle \quad \text{and} \quad \pi^i= m^i / \langle m^i, \mathbf 1 \rangle.
$$
Thanks to \eqref{E:ma} and \eqref{E:mi}, these are given explicitly  by
\begin{equation} \label{E:pia} \pi^a(k) =  c_a(1-\delta/\rho)^{k-1} {\mathrm B} (1+\alpha/\rho, k), \quad \text{for all }k\geq 1,
\end{equation}
with
$$1/c_a= \sum_{j=1}^{\infty} (1-\delta/\rho)^{j-1} {\mathrm B} (1+\alpha/\rho, j),$$
and
\begin{equation} \label{E:pii} \pi^i(k) = c_i(1-\delta/\rho)^{k-1} {\mathrm B} (\alpha/\rho, k+1), \quad \text{for all }k\geq 1,
 \end{equation}
 with
 $$1/c_i=\sum_{j=1}^{\infty} (1-\delta/\rho)^{j-1} {\mathrm B} (\alpha/\rho, j+1).$$
We can now state the convergence of the empirical distributions.

\begin{corollary}\label{C4}
Assume \eqref{E:surcrit}. Then conditionally on the event that the epidemic survives forever,
we have almost surely
$$\lim_{t\to \infty} \Pi^a(t) = \pi^a \quad\text{and} \quad \lim_{t\to \infty} \Pi^i(t) = \pi^i.$$
\end{corollary}
\begin{proof} Indeed,  recall that $W_{\infty}>0$ a.s. conditionally on survival of the epidemic.
We derive from Theorem \ref{T1} that on this event, 
$$\lim_{t\to \infty} A^f(t)/A^{\mathbf 1}(t) = \langle \pi^a, f\rangle
\quad\text{and} \quad  \lim_{t\to \infty} I^f(t)/I^{\mathbf 1}(t) = \langle \pi^i, f\rangle$$
for every bounded function $f:\N\to \R$. 
\end{proof}

In words, Corollary \ref{C4} states that conditionally on survival of the epidemic,   the empirical distributions of active cluster sizes and of isolated clusters sizes converge as time goes to infinity   to $\pi^a$ and $\pi^i$, respectively. 
 We shall therefore think of  the latter as describing asymptotically  the average distributions of the sizes of active clusters and of isolated clusters, respectively.
It is interesting to  point at similarities between Corollary \ref{C4} and earlier results by Deijfen   \cite[Theorem 1.1 and Example 2]{Deijfen}
on the asymptotic degree distribution for certain random evolving  networks. More specifically, in this model,  new vertices arrive  in continuous time, are then connected to an existing vertex with probability proportional the so-called fitness of that vertex,  and vertices further  die at rates depending on their accumulated in-degrees.
 Although the model considered by  Deijfen is different from ours, it bears also a clear resemblance, and the similarities between the results (and the methods as well) should not come as a surprise.

It is also interesting to observe from  the formulas \eqref{E:pia} and \eqref{E:pii} and
the elementary  identity
$$\frac{\alpha}{\rho} {\mathrm B} (\alpha/\rho, k+1) = k  {\mathrm B} (1+\alpha/\rho, k),$$
 that 
$$\pi^i(k) = \frac{k\pi^a(k)}{\sum_{j=1}^{\infty} j \pi^a(j)}, \quad \text{for all }k\geq 1.$$
In words, the average distribution of the sizes of isolated clusters is the size-biased of that of active clusters. This relation stems from the fact that the rate  at which an active cluster becomes isolated is proportional to its size. Since  the empirical distribution of active cluster sizes converges to $\pi^a$, the empirical distribution of isolated cluster sizes must converge to the size-biased version of $\pi^a$. We refer to Corollary 1 of \cite{Bansaye} and its proof for details of a rigorous argument.

\section{Concluding comments}

\subsection{Comparison with a model of Bansaye, Gu, and Yuan, and an eigenproblem}
 This work has been inspired by a recent manuscript by Bansaye \textit{et al.} \cite{Bansaye} in which the authors introduced a similar model for epidemics with contact tracing and cluster isolation. The main difference with the present one is that in \cite{Bansaye}, contaminations are  always traceable initially, but traceability gets lost at some fixed rate. In other words, edges in the contamination tree are traceable when they first appear, and become untraceable after an exponentially distributed time-laps, independently of the other edges.  As a consequence, clusters do not only grow  when new contamination events occur, but also split when a traceable edge becomes untraceable.

 Bansaye \textit{et al.} investigate the large time asymptotic behavior of the epidemic using different tools, namely they analyze first a deterministic eigenproblem for a growth-fragmentation-isolation equation which is naturally related to their setting; furthermore they also rely on known properties of random recursive trees. They establish results similar to our Theorem \ref{T1} and Corollary \ref{C4}  in terms of these eigenelements; the statements in \cite{Bansaye} are however less precise than ours,  as no explicit formulas for the eigenelements are given (only their existence  is established). 
 
 In our setting, using the notation of Section 4, the expectation of linear functionals of clusters at a given time   yields a family $(\nu_t, t\geq 0)$ of measures on $\N$ given by
 $$\langle \nu_t, f\rangle = \E(A^f(t)),$$
 where $f: \N\to \R_+$ is a generic bounded function. From the dynamics of the epidemic, one gets the evolution equation 
 \begin{equation} \label{E:GFE}
 \dd  \langle \nu_t, f\rangle =  \langle \nu_t, {\mathcal A}f\rangle \dd t, 
 \end{equation}
 with 
 \begin{equation} \label{E:GF}
 {\mathcal A}f(k) =k\left( p\gamma (f(k+1)-f(k)) + (1-p)\gamma f(1)- \delta f(k)\right);
 \end{equation}
 the initial condition $\nu_0$ is the Dirac mass at $1$ since we assume that the epidemic 
 starts from a single contagious individual. Specifically, in  \eqref{E:GF}, the term
 $kp\gamma (f(k+1)-f(k))$ accounts for the growth of a cluster from size $k$ to $k+1$, which occurs
 with rate  $kp\gamma$. The term $k(1-p)\gamma f(1)$ stems from the birth of new clusters of size $1$
  (i.e. an untraceable contamination) induced  by a cluster of size $k$, which occurs
 with rate  $k(1-p)\gamma$, and finally $-k\delta f(k)$ for the isolation of a cluster of size $k$, which occurs with rate  $k\delta$. 
  This formula for the  infinitesimal generator $ {\mathcal A}$  should be compared with Lemma 1 in \cite{Bansaye}, and notably Equation (4.15) there. 
 
 Predominantly,  growth-fragmentation equations (and more generally, evolution equations) cannot be solved explicitly, and most works in this area are concerned with the large time asymptotic behavior of its solutions;  see \cite{Bansaye} for some references. Roughly speaking, the paradigm, which stems from the Perron-Frobenius theorem for matrices with positive entries, is to resolve the eigenproblem for the infinitesimal generator, that is to determine 
 the principal eigenvalue (i.e. the eigenvalue with the largest real part) and its left eigenfunctions. The principal eigenvalue  is identified as the Malthusian parameter, and the left eigenfunction (viewed as a measure) yields the so-called asymptotic profile, that is, in our setting, the measure $m^a$ in Theorem \ref{T1}. So the analysis carried out in the present Section 4 solves indirectly this eigenproblem for \eqref{E:GF}, the solution being given by \eqref{E:alpha} and \eqref{E:pia}. Specifically, it holds for all bounded $f: \N\to \R_+$ that
 $$\langle \mathcal{A}^{\top}\pi^a , f \rangle \coloneqq\langle \pi^a ,\mathcal{A} f \rangle = \alpha\langle \pi^a, f\rangle.$$
However, it does not seem straightforward to check this identity directly, and we shall provide more details below.
   
Let $\nu_t(k)$ denote the expected number of active clusters of size $k\geq 1$ at time $t$. 
We have from the dynamics of the epidemic (see the discussion following \eqref{E:GF}), that for $k\geq 2$
\begin{equation} \label{E:MKvF}
\frac{\partial \nu_t(k)}{\partial t} + p\gamma\left( k \nu_t(k)- (k-1)\nu_t(k-1)\right)  =- \delta k \nu_t(k) ,
\end{equation} 
whereas for $k=1$,
\begin{equation} \label{E:MKvF1}
\frac{\partial \nu_t(1)}{\partial t} + p\gamma  \nu_t(1) = (1-p)\gamma \sum_{j=1}^{\infty} j\nu_t(j) -\delta \nu_t(1). 
\end{equation} 
Of course,  \eqref{E:MKvF} and \eqref{E:MKvF1} are equivalent to the evolution equation \eqref{E:GFE}.
From the point of view of age-structured population models (recall Remark \ref{R:1}), these
 should be viewed as a version of McKendrick-von Foerster PDE; see \cite[Equations (23.4) and (23.5)]{Kot}. 

Following \cite[Chapter 23]{Kot}, it is then natural to search for a special solution to 
\eqref{E:MKvF} and \eqref{E:MKvF1} in the form $\nu_t(k)= \e^{r t}\nu(k)$ for some $r>0$ and  some measure $\nu$ on $\N$, which of course amounts to solving the eigenproblem $ \mathcal A^{\top} \nu = r \nu$. 
Recall the notation \eqref{E:rho}; from \eqref{E:MKvF} and \eqref{E:MKvF1}, we get  first the linear recurrence equation
\begin{equation} \label{E:Eign2}
\nu(k) = \frac{p\gamma (k-1)}{r+\rho k}\nu(k-1),   \qquad k\geq 2, \end{equation} 
and then, for $k=1$, the identity
\begin{equation} \label{E:Eigen1}
(r+\rho) \nu(1) = (1-p)\gamma \sum_{j=1}^{\infty} j\nu(j).
\end{equation}
We readily deduce from \eqref{E:Eign2} and well-known properties of the beta function $\mathrm B$ that 
\begin{equation} \label{E:eigenf}
\nu(k) = c(1-\delta/\rho)^{k-1} \mathrm B(1+r/\rho,k ), \qquad k\geq 1,
\end{equation}
where $c>0$ is some arbitrary constant. We  note that $(r+\rho)\nu(1)/c=\rho$, and can now determine $r$ by rewriting \eqref{E:Eigen1} in the form:
\begin{align*}
\rho &=  (1-p)\gamma \sum_{j=1}^{\infty} j (1-\delta/\rho)^{j-1} \mathrm B(1+r/\rho,j )\\
&=  (1-p)\gamma \sum_{j=1}^{\infty} \int_0^1   j (1-\delta/\rho)^{j-1} (1-x)^{j-1} x^{r/\rho}\dd x \\
&=  (1-p)\gamma  \int_0^1  \frac{ x^{r/\rho}}{(1-(1-\delta/\rho)(1-x))^2}\dd x \\
&=  (1-p)\gamma \rho^2 \int_0^1  \frac{ x^{r/\rho}}{((\rho-\delta)x+\delta)^2}\dd x. 
\end{align*}
We have recovered the equation \eqref{E:alpha} that determines the Malthusian parameter. So  $r=\alpha$,
and we conclude from \eqref{E:eigenf} and \eqref{E:pia} that $\nu$ and $\pi^a$ are indeed proportional. 

The calculations above are reminiscent of those for the Leslie model \cite[Chapter 22]{Kot}, 
and in particular \eqref{E:alpha} can be thought of as an Euler-Lotka equation \cite[Equation (20.6)]{Kot}. 

\subsection{A detection paradox} We next discuss in further details the detection paradox which was mentioned in the introduction. Imagine that we rank the clusters in the increasing order of their birth times
rather than indexing them by the Ulam-Harris genealogical tree as we did previously. This sequence is infinite if and only if the epidemic survives, and conditionally on that event, its elements are independent, each being distributed as the typical cluster. One may then expect from the law of large numbers that, as time goes to infinity, the limit $\pi^i$ of the empirical distribution of the sizes of isolated clusters should coincide with the law of the size of a typical cluster at the time when it is detected, that is, by Lemma \ref{L1}, the geometric distribution with success probability $\delta/\rho$. However \eqref{E:pii} shows that this is not the case, and more precisely, $\pi^i$ is rather a biased version of the geometric law, where the bias is given by a beta function. 

The naive argument above has of course a flaw, which stems from  the fact that the empirical distribution of the isolated clusters at a given time corresponds to a partial sum of clusters which are listed in the increasing  order of  their detection times rather than of  their birth times. This reordering tends to list first clusters which are quickly detected and hence had little time to grow,
which hints at the feature that the average isolated cluster size is dominated stochastically by the size of a  typical isolated cluster. Nonetheless, reordering alone is not sufficient to explain the detection paradox; the second crucial ingredient is the exponential growth, and more precisely the fact that the number of clusters, say for simplicity  born during the time-interval $[t,t+1]$, is of the same order as the number of all the clusters born before time $t$, no matter how large $t$ is. A significant proportion of clusters born during  $[t,t+1]$ are detected before time $t+1$; due to the time constraint, these clusters have on average a smaller size than the typical cluster when it is isolated, and this explains the seeming paradox.

For a better understanding of the mechanisms at work in the explanation above, it may be useful to consider the following elementary example. Consider a Poisson point process on $\R_+\times (0,\infty)$ whose atoms are denoted generically by $(b,\ell)$, and which has intensity $\e^{b} \dd b \lambda(\dd \ell)$, where $\lambda$ is some probability measure on $(0,\infty)$.  We think of $(b, \ell)$ as an individual born at time $b$ and with lifespan $\ell$. Imagine that we want to estimate the lifespan distribution $\lambda$, that is, more specifically, the quantity $\langle \lambda,f\rangle$ for an arbitrary bounded continuous function $f: (0,\infty)\to \R$, from the observation of the population up to some large time $t$. If we could observe the lifespan of an individual at the time when it is born, then this would be an easy matter. Indeed, it then suffices to compute the empirical mean of $f(\ell)$ for individuals $(b,\ell)$ born at time $b\leq t$, and it is readily checked by Poissonian computation that this quantity converges almost surely to $\langle \lambda,f\rangle$
as $t\to \infty$. But of course assuming that the lifespan can be observed at the birth of an individual is unrealistic, and let us rather assume that  lifespan can be observed at  death only. 

The total number of dead individuals at time $t$ has the Poisson distribution with parameter
$$\int_0^t \e^b \lambda((0,t-b]) \dd b \sim \e^t \int_{(0,\infty)} \e^{-\ell} \lambda(\dd \ell), \qquad \text{as }t\to \infty.$$
More generally, it is easily checked that if  write $\langle M(t),f\rangle$ for the empirical mean of $f(\ell)$ computed for all individuals who are dead at time $t$, that is such that $b+\ell\leq t$, then 
$$\lim_{t\to \infty} \langle M(t),f\rangle =\langle \lambda_1,f\rangle,$$
where 
$$\lambda_1(\dd \ell) =  \frac{e^{-\ell}  \lambda(\dd \ell)}
{\int_{(0,\infty)} \e^{-s} \lambda(\dd s)}.$$
In other words, the empirical mean $\langle M(t),f\rangle$ is a consistent estimator of $ \langle \lambda_1,f\rangle$ rather of $\langle \lambda,f\rangle$. 

We stress that this detection paradox disappears for a version of this model where the intensity of the Poisson point process only
grows sub-exponentially in time, say for simplicity is given by $b^r \dd b \lambda(\dd \ell)$ for some $r>0$. The same calculation as  above easily shows that  the empirical mean of $f(\ell)$ computed for all individuals who are dead by time $ t$ then does converge to $\langle \lambda,f\rangle$ as $t\to \infty$. 

\subsection {Relation to the Yule-Simon distribution}
Following G. Udny Yule \cite{Yule}, 
 Herbert A. Simon \cite{Simon} introduced in 1955  an elementary model  depending  on a parameter $q\in(0,1)$  (that accounts for the memory of the model),  which nowadays  would be referred to as an algorithm with preferential attachment.  Simon's algorithm produces a random text, that is a long string of words $w_1\ldots w_n$, as follows.
Once the first word $w_1$ has been written, for each $j=1, \ldots,n-1$,  $w_{j+1}$ is copied from a uniform sample from $w_1, \ldots, w_j$ with probability  $q$, and with complementary probability $1-q$, $w_{j+1}$ is a new word different from all the preceding. Simon proved that for every fixed $k\geq 1$, 
the expected proportion of different words which have been written exactly $k$ times in the text
converges as $n\to \infty$ towards
\begin{equation}\label{E:YU} 
\sigma_q(k)=\frac{1}{q} \mathrm B(1+1/q, k).
\end{equation}
The probability measure  on $\N$, $\sigma_q=(\sigma_q(k), k\geq 1)$, is known as  the Yule-Simon distribution with parameter $1/q$. 
Comparing \eqref{E:pia} with \eqref{E:YU}, we can now view the average distributions of the sizes of active clusters $m^a$ as an exponentially tilted version of the Yule-Simon distribution with parameter $\alpha/\rho$. 

In this direction, we observe that  the limiting case of our model with $\delta=0$, which corresponds to a degenerate case  where detection is absent, merely rephrases Simon's algorithm
with memory parameter $q=p$. 
When there is no detection, the evolution of a typical cluster is just that of a Yule process with rate $p\gamma$ (without killing). Then the intensity measure of birth times of new clusters given by $\mu(\dd t)= (1-p) \gamma \e^{p\gamma t} \dd t$ and the Malthusian parameter can be identified by solving
$$(1-p) \gamma \int_0^{\infty} \e^{(p\gamma-\alpha)t}\dd t=1.$$
We have plainly  $\alpha = \gamma$ and the parameter of the Yule-Simon distribution is simply $\alpha/\rho=1/p$. In this setting, the degenerate case of Corollary \ref{C4} for $\delta =0$ can be viewed a strong version of Simon's result, where the converge is almost-sure and not just for the expectation.
See \cite[Example B.11]{HolJan} for a closely related discussion in the setting of Yule's original 
model of evolution of species, which is a bit different but nonetheless also yields the Yule-Simon distribution \eqref{E:YU}.

  \vskip 1cm
    \noindent\textbf{Acknowledgment.} I would like to thank Vincent Bansaye for pointing at some similarities with age-structured models and, in particular, at the existence of explicit solutions to eigenproblems for the latter. I am also grateful to two anonymous referees for their careful reading of the first version of this work and their constructive comments. 
\bibliography{tracisol.bib}

\end{document}